\documentclass[12pt]{amsart}

\usepackage{amsmath,amssymb,amsfonts}
\usepackage{mathrsfs}
\usepackage{mathtools}
\usepackage{enumitem}
\usepackage{xcolor}
\usepackage{hyperref}
\usepackage[alphabetic,msc-links]{amsrefs}
\usepackage{tikz-cd}

\definecolor{referenceblue}{RGB}{0,70,140}
\hypersetup{
  colorlinks=true,
  linkcolor=referenceblue,
  citecolor=referenceblue,
  urlcolor=referenceblue,
  pdftitle={Simplicial Volume and Scalar Curvature on Closed Kahler Surfaces},
  pdfauthor={Jie Min, Fangyang Zheng, and Bo Zhu},
  pdfsubject={Simplicial volume, scalar curvature, and Kahler geometry},
  pdfkeywords={simplicial volume, scalar curvature, Yamabe invariant, Kahler surface, canonical bundle}
}

\usepackage{geometry}
\newtheorem{theorem}{Theorem}[section]
\newtheorem{conjecture}[theorem]{Conjecture}
\newtheorem{proposition}[theorem]{Proposition}
\newtheorem{lemma}[theorem]{Lemma}
\newtheorem{corollary}[theorem]{Corollary}
\newtheorem{problem}[theorem]{Problem}
\theoremstyle{definition}
\newtheorem{definition}[theorem]{Definition}

\newtheorem{remark}[theorem]{Remark}

\numberwithin{equation}{section}
\newcommand{\Sc}{\operatorname{Sc}}
\newcommand{\Ric}{\operatorname{Ric}}
\newcommand{\vol}{\operatorname{vol}}
\newcommand{\Acal}{\mathcal A}
\newcommand{\Rcal}{\mathcal R}
\newcommand{\CP}{\mathbb{CP}}
\newcommand{\HH}{\mathbb H}
\newcommand{\eps}{\varepsilon}

\begin{document}

\title[Simplicial volume and scalar curvature]
{Simplicial Volume and Scalar Curvature on Closed K\"ahler Surfaces}
\date{\today}

\author{Jie Min}
\address[Jie Min]{Hetao Institute for Mathematics and Interdisciplinary
Sciences, Shenzhen, China}
\email{\href{mailto:jiemin.geometry@gmail.com}{jiemin.geometry@gmail.com}}

\author{Fangyang Zheng}
\address[Fangyang Zheng]{School of Mathematical Sciences, Chongqing Normal
University, Chongqing 401331, China}
\email{\href{mailto:20190045@cqnu.edu.cn}{20190045@cqnu.edu.cn}}

\author{Bo Zhu}
\address[Bo Zhu]{Yau Mathematical Sciences Center, Tsinghua University}
\email{\href{mailto:zhub@tsinghua.edu.cn}{zhub@tsinghua.edu.cn}}
\thanks{Fangyang Zheng is supported by NSFC~12471039 and 12141101, and
Bo Zhu is supported by NSFC~12501066.}


\begin{abstract}
Let \(M\) be a closed K\"ahler surface.  We prove that every Riemannian
metric \(g\) on \(M\) with \(\Sc_g\geq-\lambda^2\), where
\(\lambda\geq0\), satisfies
\[
  \lVert M\rVert\leq \frac{27}{2}\,\lambda^4\vol_g(M).
\]
This proves Gromov's quantitative
scalar-curvature--simplicial-volume conjecture for closed K\"ahler surfaces.
We also construct infinitely many non-K\"ahler symplectic
\(4\)-manifolds of general type with positive simplicial volume for which
the same estimate holds.
\end{abstract}

\maketitle
%\tableofcontents

\section{Introduction}
\label{sec:introduction}

% Scalar curvature is a local invariant of a Riemannian metric, whereas simplicial volume is a global homotopy invariant of the underlying manifold. 
% Gromov conjectured a quantitative relation between them, that is, a scalar curvature lower bound, together with the Riemannian volume, should control simplicial volume.

Let \(M\) be a closed oriented \(n\)-manifold.  For a real singular
\(n\)-chain \(c=\sum_i a_i\sigma_i\), set
\[
  \lVert c\rVert_1:=\sum_i |a_i|.
\]
The \emph{simplicial volume} of \(M\) is the \(\ell^1\)-seminorm of its real
fundamental class, namely,
\[
  \lVert M\rVert
  :=\inf\bigl\{\lVert c\rVert_1:
  c\in C_n(M;\mathbb R),\ \partial c=0,\ [c]=[M]_{\mathbb R}\bigr\},
\]
where \([M]_{\mathbb R}\in H_n(M;\mathbb R)\) denotes the real fundamental
class.  Simplicial volume is invariant under homotopy equivalences and
multiplicative under finite coverings.  In dimension four, it is unchanged
by connected sum with a closed simply connected manifold
(see \cite{gromov_bounded_cohomology_1982}*{Section~0.2, pp.~216, 218}).
Gromov conjectured that a lower scalar-curvature bound and the Riemannian
volume control this global topological invariant
(see \cite{Gromov_101}*{Section~28, p.~88}).
\begin{conjecture}
\label{conj:gromov}
For every \(n\geq2\), there exists a constant \(c_n\geq0\) such that every
closed oriented Riemannian \(n\)-manifold \((M,g)\) with
\(\Sc_g\geq-\lambda^2\), where \(\lambda\geq0\), satisfies
\begin{equation}
\label{eq:gromov-conjecture}
  \lVert M\rVert\leq c_n\lambda^n\vol_g(M).
\end{equation}
\end{conjecture}

When \(\lambda=0\), Conjecture~\ref{conj:gromov} says that a closed manifold
with nonnegative scalar curvature has zero simplicial volume.  This is known
for closed spin manifolds whose fundamental groups satisfy the strong Novikov
conjecture for the maximal group \(C^*\)-algebra
(see \cite{Lott_noncompact_PSC_2025}*{Corollary~A.5}).  Qiaochu Ma
and Guoliang Yu proved the positive-scalar-curvature case when the universal
cover is spin and \(\pi_1(M)\) has subexponential decay
(see \cite{MaYu_decay_2026}*{Theorem~1.3}).  Theorem~\ref{thm:main}
proves~\eqref{eq:gromov-conjecture} for closed K\"ahler surfaces.

\begin{theorem}
\label{thm:main}
Let \(M\) be a closed K\"ahler surface.  Then every Riemannian metric \(g\)
on \(M\) satisfies
\begin{equation}
\label{eq:main-integrated}
  \lVert M\rVert
  \leq \frac{27}{2}
  \int_M (\Sc_g^-)^2\,d\vol_g,
  \qquad
  \Sc_g^-:=\max\{-\Sc_g,0\}.
\end{equation}
In particular, if \(\Sc_g\geq-\lambda^2\) for some \(\lambda\geq0\), then
\begin{equation}
\label{eq:main-pointwise}
  \lVert M\rVert
  \leq \frac{27}{2}\,\lambda^4\vol_g(M).
\end{equation}
\end{theorem}

The coefficient \(27/2\) is uniform over all closed K\"ahler surfaces,
including blow-ups.  It is not optimal on every fixed manifold:
Proposition~\ref{prop:bidisk} gives the optimal coefficient
\(3/(32\pi^2)\) for closed bidisk quotients and their blow-ups.

In addition, we also establish Conjecture \ref{conj:gromov} for infinitely many non-K\"ahler symplectic $4$-manifolds of general type. These symplectic manifolds have minimal models with similar topology to K\"ahler surfaces, which made our technique viable.
\subsection*{Outline of the proof}

Lemma~\ref{lem:yamabe-reformulation} replaces the pointwise
scalar-curvature hypothesis by the metric-independent scalar cost
\[
  \Acal_4(M)
  =\bigl|\min\{\sigma(M),0\}\bigr|^2
  =\inf_g\int_M(\Sc_g^-)^2\,d\vol_g.
\]
where \(\sigma(M)\) is the Yamabe invariant of \(M\).  It is therefore
enough to prove
\[
  \lVert M\rVert\leq\frac{27}{2}\Acal_4(M).
\]

If \(M\) is not of general type, the collapse theorem of Paternain and
Petean and Gromov's Ricci-curvature estimate imply
\(\lVert M\rVert=0\).  If \(M\) is of general type with minimal model
\(S\), then blow-up invariance of simplicial volume and LeBrun's formula
give
\[
  \lVert M\rVert=\lVert S\rVert,
  \qquad
  \Acal_4(M)=32\pi^2c_1^2(S).
\]
The canonical bundle of \(S\) is nef.  By constructing twisted negative
K\"ahler--Einstein metrics in classes approaching the canonical class and
applying Gromov's Ricci-curvature estimate,
Proposition~\ref{prop:nef-canonical} gives
\[
  \lVert S\rVert\leq432\pi^2c_1^2(S).
\]
Comparing the last two displays gives
\[
  \lVert M\rVert
  \leq\frac{432\pi^2}{32\pi^2}\Acal_4(M)
  =\frac{27}{2}\Acal_4(M).
\]
The variational characterization of \(\Acal_4(M)\) now gives
\eqref{eq:main-integrated} for every Riemannian metric \(g\).  If
\(\Sc_g\geq-\lambda^2\), then \(\Sc_g^-\leq\lambda^2\), and
\eqref{eq:main-pointwise} follows.

For symplectic \(4\)-manifolds, LeBrun's monopole-class estimate replaces
his exact formula for the Yamabe invariant of a complex surface.
Proposition~\ref{prop:symplectic} combines this estimate with the preceding
K\"ahler bound to get bound on simplicial volume, while Proposition~\ref{prop:symplectic-exmaple} constructs an infinite non-K\"ahler family to which the simplicial volume bound holds.

The nef-canonical estimate holds in every complex dimension.  In complex
dimension two, LeBrun's formula identifies canonical volume with scalar
cost.  Proposition~\ref{prop:higher-yamabe-failure} shows that this identity
fails in every complex dimension at least three: there are closed projective
manifolds with ample canonical bundle and vanishing scalar cost.

\section{Yamabe reduction and nef canonical classes}
\label{sec:scalar-nef}

\subsection{The Yamabe reformulation}

Let \(M^n\) be a closed smooth manifold of dimension \(n\geq3\).
For a Riemannian metric \(g\), let \(\Sc_g\) denote its scalar curvature.
Every metric \(\widehat g\) in the conformal class \([g]\) can be written
uniquely as
\[
  \widehat g=u^{4/(n-2)}g
  \qquad\text{for some }u\in C^\infty(M),\quad u>0.
\]
Use the convention \(\Delta_g=\operatorname{div}_g\nabla_g\), so that
\(\int_M-u\Delta_gu\,d\vol_g=\int_M|\nabla_gu|^2\,d\vol_g\).
The conformal transformation laws are
\[
  \Sc_{\widehat g}
  =u^{-(n+2)/(n-2)}
    \left(-\frac{4(n-1)}{n-2}\Delta_gu+\Sc_gu\right),
  \qquad
  d\vol_{\widehat g}=u^{2n/(n-2)}\,d\vol_g.
\]
Consequently,
\begin{align*}
  \int_M\Sc_{\widehat g}\,d\vol_{\widehat g}
  &=\int_M\left(-\frac{4(n-1)}{n-2}u\Delta_gu+\Sc_gu^2\right)
    \,d\vol_g\\
  &=\int_M\left(\frac{4(n-1)}{n-2}|\nabla_gu|^2+\Sc_gu^2\right)
    \,d\vol_g.
\end{align*}
Moreover,
\[
  \vol_{\widehat g}(M)=\int_Mu^{2n/(n-2)}\,d\vol_g.
\]
The normalized total scalar curvature of \(\widehat g\) is therefore the
quotient below.  The coefficient \(4(n-1)/(n-2)\) comes from the conformal
transformation law for scalar curvature.  Define the Yamabe constant of
\([g]\) by
\begin{equation}
\label{eq:yamabe-constant-def}
  Y(M,[g])
  :=\inf_{\substack{u\in C^\infty(M)\\ u>0}}
  \frac{\displaystyle\int_M
    \left(\frac{4(n-1)}{n-2}|\nabla_g u|^2+\Sc_g u^2\right)
    \,d\vol_g}
       {\displaystyle
        \left(\int_M u^{2n/(n-2)}\,d\vol_g\right)^{(n-2)/n}}.
\end{equation}
The smooth Yamabe invariant of \(M\) is then
\begin{equation}
\label{eq:yamabe-invariant-def}
  \sigma(M):=\sup_{[g]}Y(M,[g]),
\end{equation}
where the supremum is taken over all conformal classes on \(M\).  For a
Riemannian metric \(g\), write \(\Sc_g^-:=\max\{-\Sc_g,0\}\), and define
\begin{equation}
\label{eq:scalar-cost-def}
  \Acal_n(M)
  :=\inf\{\vol_g(M):\Sc_g\geq-1\}.
\end{equation}
We call \(\Acal_n(M)\) the \emph{scalar cost} of \(M\).

\begin{lemma}
\label{lem:yamabe-reformulation}
Let \(M\) be a closed smooth \(n\)-manifold with \(n\geq3\).  Then
\begin{equation}
\label{eq:yamabe-three-functionals}
  \Acal_n(M)
  =\inf_g\int_M(\Sc_g^-)^{n/2}\,d\vol_g
  =\bigl|\min\{\sigma(M),0\}\bigr|^{n/2}.
\end{equation}
Moreover, for any fixed constant \(c_n\geq0\), the following assertions are
equivalent:
\begin{enumerate}[label=\textup{(\roman*)}]
\item For every closed oriented Riemannian \(n\)-manifold \((N,g)\) and every
\(\lambda\geq0\),
\[
  \Sc_g\geq-\lambda^2
  \quad\Longrightarrow\quad
  \lVert N\rVert\leq c_n\lambda^n\vol_g(N).
\]
\item For every closed oriented \(n\)-manifold \(N\),
\begin{equation}
\label{eq:minvol-gromov-equivalent}
  \lVert N\rVert\leq c_n\Acal_n(N).
\end{equation}
\item For every closed oriented \(n\)-manifold \(N\),
\begin{equation}
\label{eq:yamabe-gromov-equivalent}
  \lVert N\rVert
  \leq c_n\bigl|\min\{\sigma(N),0\}\bigr|^{n/2}.
\end{equation}
\item For every closed oriented Riemannian \(n\)-manifold \((N,g)\),
\begin{equation}
\label{eq:integrated-gromov-equivalent}
  \lVert N\rVert
  \leq c_n\int_N(\Sc_g^-)^{n/2}\,d\vol_g.
\end{equation}
\end{enumerate}
\end{lemma}

\begin{proof}
Fix a Riemannian metric \(g\) on \(M\).  For every positive smooth
function \(u\), H\"older's inequality gives
\begin{align*}
 &\int_M
   \left(\frac{4(n-1)}{n-2}|\nabla_gu|^2+\Sc_gu^2\right)\,d\vol_g\\
 &\qquad\geq-\int_M\Sc_g^-u^2\,d\vol_g\\
 &\qquad\geq
 -\left(\int_M(\Sc_g^-)^{n/2}\,d\vol_g\right)^{2/n}
  \left(\int_Mu^{2n/(n-2)}\,d\vol_g\right)^{(n-2)/n}.
\end{align*}
Dividing by the last factor and taking the infimum over \(u\) yields
\begin{equation}
\label{eq:yamabe-negative-part}
  Y(M,[g])\geq
  -\left(\int_M(\Sc_g^-)^{n/2}\,d\vol_g\right)^{2/n}.
\end{equation}

Write
\[
  I_n(M):=\inf_g\int_M(\Sc_g^-)^{n/2}\,d\vol_g.
\]
We distinguish two cases.
\begin{itemize}
\item If \(\sigma(M)\leq0\), then \(Y(M,[g])\leq\sigma(M)\), and
\eqref{eq:yamabe-negative-part} implies
\[
  \left(\int_M(\Sc_g^-)^{n/2}\,d\vol_g\right)^{2/n}
  \geq -Y(M,[g])
  \geq -\sigma(M).
\]
Consequently, for every metric \(g\),
\[
  \int_M(\Sc_g^-)^{n/2}\,d\vol_g
  \geq |\sigma(M)|^{n/2}.
\]
Hence \(I_n(M)\geq |\sigma(M)|^{n/2}\).  Choose conformal classes
\([g_j]\) such that \(Y(M,[g_j])\to\sigma(M)\).  By the solution of the
Yamabe problem, each \([g_j]\) contains a unit-volume Yamabe minimizer
\(h_j\) (see \cite{AkutagawaIshidaLeBrun_2007}*{pp.~72--73}).  Its scalar
curvature is constant and equals \(Y(M,[g_j])\).  Thus
\[
  \int_M(\Sc_{h_j}^-)^{n/2}\,d\vol_{h_j}
  =|Y(M,[g_j])|^{n/2}
  \longrightarrow |\sigma(M)|^{n/2}.
\]
Therefore \(I_n(M)=|\sigma(M)|^{n/2}\).

\item If \(\sigma(M)>0\), then some conformal class has positive Yamabe
constant.  Its Yamabe minimizer has positive scalar curvature, so
\(I_n(M)=0\).
\end{itemize}
Consequently,
\begin{equation}
\label{eq:negative-part-yamabe}
  I_n(M)=\bigl|\min\{\sigma(M),0\}\bigr|^{n/2}.
\end{equation}

If \(\Sc_g\geq-1\), then \(\Sc_g^-\leq1\), so
\[
  \int_M(\Sc_g^-)^{n/2}\,d\vol_g\leq\vol_g(M),
\]
and taking the infimum over all such metrics gives
\(I_n(M)\leq\Acal_n(M)\).  For the reverse inequality, we again distinguish
cases according to the sign of \(\sigma(M)\).
\begin{itemize}
\item If \(\sigma(M)<0\), then for all large \(j\), the metric
\[
  \widetilde h_j
  :=|Y(M,[g_j])|h_j
\]
has scalar curvature \(-1\) and volume
\[
  \vol_{\widetilde h_j}(M)
  =|Y(M,[g_j])|^{n/2}
  \longrightarrow |\sigma(M)|^{n/2}.
\]
Here and below we use the scaling laws
\[
  \Sc_{ag}=a^{-1}\Sc_g,
  \qquad
  \vol_{ag}(M)=a^{n/2}\vol_g(M),
  \qquad a>0.
\]

\item Suppose that \(\sigma(M)=0\).  If \(Y(M,[g_j])<0\) along a
subsequence, the same rescaling produces metrics with scalar curvature
\(-1\) and volumes tending to zero.  Otherwise, \(Y(M,[g_j])=0\) for all
large \(j\); constant rescalings of a corresponding scalar-flat Yamabe
minimizer then have arbitrarily small volume.

\item If \(\sigma(M)>0\), choose a positive-scalar-curvature metric \(h\).
The metrics \(ah\), with \(a\to0\), still have scalar curvature bounded
below by \(-1\), while their volumes tend to zero.
\end{itemize}
Thus, in every case,
\[
  \Acal_n(M)
  \leq\bigl|\min\{\sigma(M),0\}\bigr|^{n/2}
  =I_n(M),
\]
where the equality follows from \eqref{eq:negative-part-yamabe}.  Together
with \(I_n(M)\leq\Acal_n(M)\), this proves
\eqref{eq:yamabe-three-functionals}.

It remains to prove the equivalence of \textup{(i)}--\textup{(iv)}.
Applying \textup{(i)} with \(\lambda=1\) and taking the infimum over all
metrics with \(\Sc_g\geq-1\) gives \textup{(ii)}.  Identity
\eqref{eq:yamabe-three-functionals} makes \textup{(ii)} and \textup{(iii)}
equivalent.  It also gives, for every Riemannian metric \(g\) on \(N\),
\[
  \Acal_n(N)
  \leq\int_N(\Sc_g^-)^{n/2}\,d\vol_g,
\]
so \textup{(ii)} implies \textup{(iv)}.  Finally, if
\(\Sc_g\geq-\lambda^2\), then \(\Sc_g^-\leq\lambda^2\), and hence
\[
  \int_N(\Sc_g^-)^{n/2}\,d\vol_g
  \leq\lambda^n\vol_g(N).
\]
Thus \textup{(iv)} implies \textup{(i)}, completing the cycle of
implications.
\end{proof}

In dimension four, Lemma~\ref{lem:yamabe-reformulation} gives, for every
Riemannian metric \(g\),
\begin{equation}
\label{eq:A4-negative-part}
  \Acal_4(M)=\bigl|\min\{\sigma(M),0\}\bigr|^2
  \leq\int_M(\Sc_g^-)^2\,d\vol_g.
\end{equation}
The main theorem reduces to the metric-independent estimate
\(\lVert M\rVert\leq C\Acal_4(M)\).

\subsection{Nef canonical bundles and Ricci lower bounds}

A K\"ahler form determines a Riemannian metric, its Ricci form represents
the first Chern class, and the top power of its cohomology class computes
the Riemannian volume.

Let \(X\) be a closed K\"ahler manifold of complex dimension \(m\), hence of
real dimension \(2m\), and use its complex orientation.  Its
\emph{canonical bundle} is the holomorphic line bundle
\begin{equation}
\label{eq:canonical-bundle-def}
  K_X:=\Lambda^m(T^{1,0}X)^*=\Lambda^{m,0}T^*X.
\end{equation}
Thus, in holomorphic coordinates, a local section of \(K_X\) has the form
\[
  f\,dz^1\wedge\cdots\wedge dz^m,
\]
where \(f\) is holomorphic.
We write \(c_1(X):=c_1(T^{1,0}X)\).  Since \(K_X\) is the dual of the
determinant line bundle of \(T^{1,0}X\),
\begin{equation}
\label{eq:canonical-first-chern}
  c_1(K_X)=-c_1(X).
\end{equation}
All first Chern classes below are viewed in real de Rham cohomology.

A K\"ahler form \(\omega\) is a closed real \(2\)-form satisfying
\[
  \omega(Ju,Jv)=\omega(u,v),
  \qquad
  \omega(u,Ju)>0\quad\text{for \(u\neq0\)}.
\]
It determines a Riemannian metric \(g_\omega\) by
\[
  \omega(u,v)=g_\omega(Ju,v),
  \qquad
  g_\omega(u,v)=\omega(u,Jv)
\]
for real tangent vectors \(u,v\).  In holomorphic coordinates, write
\(\omega=i g_{j\bar k}\,dz^j\wedge d\bar z^k\), where
\((g_{j\bar k})\) is a positive-definite Hermitian matrix.  The de Rham
cohomology class \([\omega]\in H^{1,1}(X;\mathbb R)\) is its
\emph{K\"ahler class}.  The Ricci form of \(\omega\) is
\begin{equation}
\label{eq:ricci-form-convention}
  \rho(\omega)
  :=-i\partial\bar\partial\log\det(g_{j\bar k}),
  \qquad
  [\rho(\omega)]=2\pi c_1(X)=-2\pi c_1(K_X).
\end{equation}
The real Ricci tensor of \(g_\omega\) and the Ricci form are related by
\begin{equation}
\label{eq:ricci-form-tensor-dictionary}
  \rho(\omega)(u,v)=\Ric_{g_\omega}(Ju,v).
\end{equation}
Because the Ricci tensor of a K\"ahler metric is \(J\)-invariant, the form
inequality \(\rho(\omega)\geq-\lambda\omega\) is equivalent to
\(\Ric_{g_\omega}\geq-\lambda g_\omega\).  Evaluating the form inequality
on \((u,Ju)\) gives
\[
  \Ric_{g_\omega}(u,u)
  =\rho(\omega)(u,Ju)
  \geq-\lambda\omega(u,Ju)
  =-\lambda g_\omega(u,u).
\]
Finally, if \(\alpha=[\omega]\in H^{1,1}(X;\mathbb R)\) is the K\"ahler
class, then the complex orientation gives
\begin{equation}
\label{eq:kahler-volume}
  d\vol_{g_\omega}=\frac{\omega^m}{m!},
  \qquad
  \vol_{g_\omega}(X)=\frac1{m!}\int_X\alpha^m.
\end{equation}
Thus the top self-intersection of \(\alpha\) determines the volume of the
associated metric.  Here and below,
\(\int_X\alpha^m\) means evaluation of the top-degree cohomology class
\(\alpha^m\) on the fundamental class of \(X\).

We use the complex Monge--Amp\`ere equation.  Recall that
\(d=\partial+\bar\partial\) and that, for a
real-valued function \(f\), the form \(i\partial\bar\partial f\) is real and
exact.  Hence \(\widehat\omega+i\partial\bar\partial f\), when positive,
is a K\"ahler form in the same cohomology class as \(\widehat\omega\).

\begin{lemma}
\label{lem:twisted-negative}
Let \(X\) be a closed K\"ahler manifold, let \(\alpha\) be a K\"ahler
class, and let \(\lambda>0\).  Let \(\theta\) be a smooth closed real
\((1,1)\)-form that is semipositive, meaning that
\(\theta(u,Ju)\geq0\) for every real tangent vector \(u\).  Suppose that
the following identity holds in real de Rham cohomology:
\begin{equation}
\label{eq:twisted-class-identity}
  2\pi c_1(X)=-\lambda\alpha+[\theta].
\end{equation}
Then there exists a K\"ahler form \(\omega\) with \([\omega]=\alpha\)
such that
\begin{equation}
\label{eq:twisted-ricci-equation}
  \rho(\omega)=-\lambda\omega+\theta.
\end{equation}
In particular, the associated real Riemannian metric \(g_\omega\) satisfies
\begin{equation}
\label{eq:twisted-real-ricci-bound}
  \Ric_{g_\omega}\geq-\lambda g_\omega.
\end{equation}
\end{lemma}

\begin{proof}
Choose a K\"ahler form \(\widehat\omega\) with
\([\widehat\omega]=\alpha\).  By
\eqref{eq:ricci-form-convention}, the cohomological identity
\eqref{eq:twisted-class-identity} gives
\[
  [\rho(\widehat\omega)+\lambda\widehat\omega-\theta]
  =2\pi c_1(X)+\lambda\alpha-[\theta]=0.
\]
Hence \(\rho(\widehat\omega)+\lambda\widehat\omega-\theta\) is a real exact
\((1,1)\)-form.  On a closed K\"ahler manifold, the
\(\partial\bar\partial\)-lemma says that every such form is
\(i\partial\bar\partial\) of a real-valued function.  Hence there is a
smooth real function \(F\) such that
\begin{equation}
\label{eq:twisted-potential}
  \rho(\widehat\omega)+\lambda\widehat\omega-\theta
  =i\partial\bar\partial F.
\end{equation}
The Aubin--Yau theorem in the negative case solves
\begin{equation}
\label{eq:negative-ma}
  (\widehat\omega+i\partial\bar\partial\varphi)^m
  =e^{\lambda\varphi+F}\widehat\omega^m.
\end{equation}
It provides a smooth real function \(\varphi\) such that
\(\widehat\omega+i\partial\bar\partial\varphi\) is positive and
\eqref{eq:negative-ma} holds.  The sign \(+\lambda\varphi\) is the negative
Einstein sign; because \(\lambda>0\), no integral normalization of \(F\) is
needed (see \cite{Aubin_Monge_Ampere_1978} and
\cite{Yau_Ricci_Kahler_1978}).

Set
\[
  \omega:=\widehat\omega+i\partial\bar\partial\varphi.
\]
Positivity makes \(\omega\) a K\"ahler form, while exactness of
\(i\partial\bar\partial\varphi\) gives
\([\omega]=[\widehat\omega]=\alpha\).  From the definition of the Ricci
form,
\begin{equation}
\label{eq:ricci-volume-ratio}
  \rho(\omega)-\rho(\widehat\omega)
  =-i\partial\bar\partial
    \log\frac{\omega^m}{\widehat\omega^m}.
\end{equation}
On the other hand, taking the logarithm of
\eqref{eq:negative-ma} gives
\[
  \log\frac{\omega^m}{\widehat\omega^m}
  =\lambda\varphi+F.
\]
Substituting this identity into \eqref{eq:ricci-volume-ratio} and then using
\eqref{eq:twisted-potential}, we obtain
\begin{align*}
  \rho(\omega)
  &=\rho(\widehat\omega)
    -i\partial\bar\partial(\lambda\varphi+F)\\
  &=-\lambda
    (\widehat\omega+i\partial\bar\partial\varphi)+\theta\\
  &=-\lambda\omega+\theta.
\end{align*}
To obtain the Riemannian estimate, let \(u\) be a real tangent vector.  The
\(J\)-invariance of the Ricci tensor gives
\[
\begin{aligned}
  \Ric_{g_\omega}(u,u)
  &=\rho(\omega)(u,Ju)\\
  &=-\lambda\omega(u,Ju)+\theta(u,Ju)\\
  &\geq-\lambda g_\omega(u,u).
\end{aligned}
\]
Hence \(\Ric_{g_\omega}\geq-\lambda g_\omega\).
\end{proof}

To apply the lemma, consider the K\"ahler cone
\[
  \mathcal K_X
  :=\bigl\{[\omega]\in H^{1,1}(X;\mathbb R):
       \omega\text{ is a K\"ahler form}\bigr\}.
\]
A real \((1,1)\)-class \(\beta\) is \emph{nef} if
\(\beta\in\overline{\mathcal K_X}\).  Equivalently, after fixing any
K\"ahler form \(\omega_0\),
\begin{equation}
\label{eq:nef-characterization}
  \beta+\eps[\omega_0]\in\mathcal K_X
  \qquad\text{for every \(\eps>0\)}.
\end{equation}
The canonical bundle \(K_X\) is nef when \(c_1(K_X)\) is nef.  This does
not require \(c_1(K_X)\) itself to contain a smooth semipositive form;
each perturbed class \(c_1(K_X)+\eps[\omega_0]\) is K\"ahler.  We use these
classes for \(\eps>0\).

When \(K_X\) is nef, its \emph{canonical volume} is
\begin{equation}
\label{eq:canonical-volume-def}
  \int_Xc_1(K_X)^m.
\end{equation}
This number is nonnegative, since
\begin{equation}
\label{eq:canonical-volume-limit}
  \int_Xc_1(K_X)^m
  =\lim_{\eps\downarrow0}
   \int_X\bigl(c_1(K_X)+\eps[\omega_0]\bigr)^m,
\end{equation}
and every class in the limit is K\"ahler.

The Riemannian estimate is Gromov's Main Inequality.  If
\((N,g)\) is a closed oriented Riemannian \(d\)-manifold, \(d\geq2\), with
\(\Ric_g\geq-(d-1)g\), then
\begin{equation}
\label{eq:gromov-ricci-bound}
  \lVert N\rVert\leq d!(d-1)^d\vol_g(N).
\end{equation}
To see the normalization, set \(h:=(d-1)^2g\).  A constant rescaling does
not change the Ricci tensor as a \((0,2)\)-tensor, and hence
\[
  \Ric_h=\Ric_g\geq-(d-1)g=-\frac1{d-1}h.
\]
Gromov's inequality in this normalization gives
\(\lVert N\rVert\leq C_d\vol_h(N)\), where \(C_d\leq d!\)
(see \cite{gromov_bounded_cohomology_1982}*{Section~0.5, Main Inequality}).
Since \(\vol_h(N)=(d-1)^d\vol_g(N)\), replacing \(C_d\) by \(d!\) gives
\eqref{eq:gromov-ricci-bound}.

For a complex \(m\)-manifold, the real dimension is \(d=2m\).  Set
\(\lambda=2m-1\); the metrics constructed by
Lemma~\ref{lem:twisted-negative} then satisfy the hypothesis of Gromov's
estimate.  The case \(\lambda=1\) appears in the work of
Damin Wu and Shing-Tung Yau (see
\cite{WuYau_negative_holomorphic_2016}*{Proposition~8\textup{(i)}}).

\begin{proposition}
\label{prop:nef-canonical}
Let \(X\) be a closed K\"ahler manifold of complex dimension \(m\geq1\) whose
canonical bundle is nef.  Then
\begin{equation}
\label{eq:nef-canonical}
  \lVert X\rVert
  \leq
  \frac{(2m)!}{m!}(2\pi)^m(2m-1)^m
  \int_Xc_1(K_X)^m.
\end{equation}
In particular, if the canonical volume vanishes, then
\(\lVert X\rVert=0\).
\end{proposition}

\begin{proof}
Set
\[
  \lambda:=2m-1,
\]
the constant in Gromov's real \(2m\)-dimensional Ricci bound.
Fix a K\"ahler form \(\omega_0\).  Since \(K_X\) is nef,
\eqref{eq:nef-characterization} shows that
\(c_1(K_X)+\eps[\omega_0]\) is a K\"ahler class for every \(\eps>0\).

For \(\eps>0\), define the K\"ahler class
\begin{equation}
\label{eq:alpha-epsilon}
  \alpha_\eps
  :=\frac{2\pi}{\lambda}
    \bigl(c_1(K_X)+\eps[\omega_0]\bigr)
\end{equation}
and the smooth closed positive \((1,1)\)-form
\[
  \theta_\eps:=2\pi\eps\omega_0.
\]
The factor \(2\pi\) matches the normalization
\([\rho(\omega)]=2\pi c_1(X)\) in
\eqref{eq:ricci-form-convention}.  Using
\(c_1(X)=-c_1(K_X)\), we compute
\begin{align*}
  -\lambda\alpha_\eps+[\theta_\eps]
  &=-2\pi\bigl(c_1(K_X)+\eps[\omega_0]\bigr)
    +2\pi\eps[\omega_0]\\
  &=-2\pi c_1(K_X)
   =2\pi c_1(X),
\end{align*}
so Lemma~\ref{lem:twisted-negative} applies to
\((\alpha_\eps,\theta_\eps)\).  It gives a K\"ahler form \(\omega_\eps\)
with \([\omega_\eps]=\alpha_\eps\).  For its associated real Riemannian
metric \(g_\eps:=g_{\omega_\eps}\), the conclusion of the lemma is
\[
  \Ric_{g_\eps}\geq-(2m-1)g_\eps.
\]

Applying \eqref{eq:gromov-ricci-bound} in real dimension \(2m\) gives
\begin{equation}
\label{eq:gromov-epsilon}
  \lVert X\rVert
  \leq (2m)!(2m-1)^{2m}\vol_{g_\eps}(X).
\end{equation}
The K\"ahler volume identity \eqref{eq:kahler-volume} and
\([\omega_\eps]=\alpha_\eps\) give
\begin{align*}
  \vol_{g_\eps}(X)
  &=\frac1{m!}\int_X\omega_\eps^m
    =\frac1{m!}\int_X\alpha_\eps^m\\
  &=\frac1{m!}\left(\frac{2\pi}{2m-1}\right)^m
    \int_X\bigl(c_1(K_X)+\eps[\omega_0]\bigr)^m.
\end{align*}
The last integral depends only on cohomology, and
\begin{equation}
\label{eq:canonical-intersection-expansion}
  \int_X\bigl(c_1(K_X)+\eps[\omega_0]\bigr)^m
  =\sum_{j=0}^m\binom mj\eps^j
    \int_Xc_1(K_X)^{m-j}[\omega_0]^j.
\end{equation}
Hence
\[
  \lim_{\eps\downarrow0}
  \int_X\bigl(c_1(K_X)+\eps[\omega_0]\bigr)^m
  =\int_Xc_1(K_X)^m.
\]
The left-hand side of \eqref{eq:gromov-epsilon} is independent of
\(\eps\).  We may therefore substitute the volume formula and let
\(\eps\downarrow0\), obtaining
\[
\begin{aligned}
  \lVert X\rVert
  &\leq (2m)!(2m-1)^{2m}\frac1{m!}
     \left(\frac{2\pi}{2m-1}\right)^m
     \int_Xc_1(K_X)^m\\
  &=\frac{(2m)!}{m!}(2\pi)^m(2m-1)^m
     \int_Xc_1(K_X)^m.
\end{aligned}
\]
If the canonical volume vanishes, the right-hand side is zero, and hence
\(\lVert X\rVert=0\).  The argument takes a limit only of the
cohomological volumes \(\int_X\alpha_\eps^m\); it requires no convergence
of the metrics \(g_\eps\).
\end{proof}

\section{Closed K\"ahler surfaces and bidisk quotients}
\label{sec:surfaces}

We use Section~\ref{sec:scalar-nef} and the classification of closed
K\"ahler surfaces.

\subsection{General type and scalar cost}

Let \(M\) be a closed K\"ahler surface, and let \(K_M\) be its canonical
bundle.  For each integer \(\ell\geq1\), the \(\ell\)-th
\emph{plurigenus} is
\[
  P_\ell(M):=\dim_{\mathbb C}H^0(M,K_M^{\otimes\ell}),
\]
where \(H^0(M,K_M^{\otimes\ell})\) is the vector space of holomorphic
sections of the \(\ell\)-th tensor power of \(K_M\).  If all the
plurigenera vanish, the Kodaira dimension of \(M\) is
\(\kappa(M)=-\infty\).  Otherwise,
\[
  \kappa(M)
  :=\limsup_{\ell\to\infty}
    \frac{\log P_\ell(M)}{\log\ell}.
\]
with the convention \(\log0=-\infty\).
For a K\"ahler surface, \(\kappa(M)\in\{-\infty,0,1,2\}\).  The surface is
\emph{of general type} if \(\kappa(M)=2\), or equivalently, if its
plurigenera have quadratic growth
(see \cite{LeBrun_kodaira_yamabe_1999}*{pp.~133--134}).

Kodaira dimension is invariant under blow-up.  A \emph{\((-1)\)-curve} is a
holomorphically embedded copy of \(\CP^1\) with self-intersection \(-1\),
and a complex surface is \emph{minimal} if it contains no
\((-1)\)-curves.  If \(M\) is of general type, successively blowing down
its \((-1)\)-curves produces its \emph{minimal model} \(S\), which is
unique up to biholomorphism.  The canonical bundle \(K_S\) is nef and big;
for a nef line bundle on a surface, bigness is equivalent to positive
self-intersection.  Hence
\[
  c_1^2(S):=\int_S c_1(K_S)^2>0.
\]
This is the canonical self-intersection of \(S\).  Since
\(c_1(K_S)=-c_1(T^{1,0}S)\), it also agrees with the usual Chern number
\(\int_S c_1(T^{1,0}S)^2\).

\begin{lemma}
\label{lem:surface-reduction}
Let \(M\) be a closed K\"ahler surface.
\begin{enumerate}[label=\textup{(\roman*)}]
\item If \(M\) is not of general type, then \(\lVert M\rVert=0\).
\item If \(M\) is of general type and \(S\) is its minimal model, then
\begin{equation}
\label{eq:surface-reduction-data}
  \lVert M\rVert=\lVert S\rVert,
  \qquad
  \sigma(M)=-4\pi\sqrt{2c_1^2(S)},
  \qquad
  \Acal_4(M)=32\pi^2c_1^2(S)=\Acal_4(S).
\end{equation}
\end{enumerate}
In particular, \(\lVert M\rVert>0\) forces \(M\) to be of general type,
whereas \(\sigma(M)\geq0\) forces \(\lVert M\rVert=0\).
\end{lemma}

\begin{proof}
Paternain and Petean proved that every closed K\"ahler surface that is not
of general type collapses with sectional curvature bounded from below
(see \cite{PaternainPetean_complex_surfaces_2004}*{Theorem~A and the
following Corollary}).  After rescaling the collapsing metrics, we obtain
a sequence \(g_j\) such that
\[
  \sec_{g_j}\geq-1,
  \qquad
  \vol_{g_j}(M)\longrightarrow0.
\]
In real dimension four this gives \(\Ric_{g_j}\geq-3g_j\).  Applying
\eqref{eq:gromov-ricci-bound} with \(d=4\) and then letting \(j\to\infty\)
gives
\[
  0\leq\lVert M\rVert
  \leq4!\,3^4\vol_{g_j}(M)\longrightarrow0,
\]
and hence \(\lVert M\rVert=0\).

Suppose now that \(M\) is of general type.  It is obtained from its minimal
model \(S\) by finitely many blow-ups, so
\(M\cong S\#r\,\overline{\CP}^{2}\) for some \(r\geq0\).  Simplicial
volume is unchanged by connected sum with the simply connected manifold
\(\overline{\CP}^{2}\), and hence
\(\lVert M\rVert=\lVert S\rVert\).  LeBrun's formula for a surface of
general type and its blow-ups
(see \cite{LeBrun_kodaira_yamabe_1999}*{Theorem~2}) gives
\begin{equation}
\label{eq:lebrun-yamabe}
  \sigma(M)=-4\pi\sqrt{2c_1^2(S)}<0.
\end{equation}
Lemma~\ref{lem:yamabe-reformulation} therefore yields
\[
  \Acal_4(M)=|\sigma(M)|^2
  =32\pi^2c_1^2(S)=\Acal_4(S),
\]
which proves~\textup{(ii)}.  Part~\textup{(i)} and
\eqref{eq:surface-reduction-data} give the last two implications.
\end{proof}

LeBrun's formula determines the scalar cost of a surface of general type;
Proposition~\ref{prop:nef-canonical} gives the corresponding upper bound for
simplicial volume.

\begin{proof}[Proof of Theorem~\ref{thm:main}]
If \(M\) is not of general type, then
Lemma~\ref{lem:surface-reduction} gives \(\lVert M\rVert=0\), and both
inequalities in Theorem~\ref{thm:main} follow because their right-hand
sides are nonnegative.

Suppose that \(M\) is of general type, and let
\(S\) be its minimal model.  The canonical bundle \(K_S\) is nef and big.
Applying Proposition~\ref{prop:nef-canonical} with \(m=2\) gives
\begin{equation}\label{eq:SV_boundby_K}
  \lVert S\rVert
  \leq \frac{4!}{2!}(2\pi)^2 3^2 c_1^2(S) =432\pi^2c_1^2(S).
\end{equation}
Using \eqref{eq:surface-reduction-data}, we obtain
\[
  \lVert M\rVert=\lVert S\rVert
  \leq432\pi^2c_1^2(S)
  =\frac{27}{2}\Acal_4(M).
\]
For an arbitrary Riemannian metric \(g\) on \(M\),
\eqref{eq:A4-negative-part} now gives
\[
  \lVert M\rVert
  \leq\frac{27}{2}\int_M(\Sc_g^-)^2\,d\vol_g,
\]
which is \eqref{eq:main-integrated}.  If
\(\Sc_g\geq-\lambda^2\), then \(\Sc_g^-\leq\lambda^2\), and therefore
\[
  \int_M(\Sc_g^-)^2\,d\vol_g
  \leq\lambda^4\vol_g(M).
\]
This is \eqref{eq:main-pointwise}.
\end{proof}

For a closed four-manifold with \(\Acal_4(M)>0\), define
\begin{equation}
\label{eq:R4-def}
  \Rcal_4(M):=\frac{\lVert M\rVert}{\Acal_4(M)}.
\end{equation}
Lemma~\ref{lem:yamabe-reformulation} identifies \(\Rcal_4(M)\) with the
smallest coefficient in Gromov's scalar-curvature inequality for the fixed
smooth manifold \(M\).  Moreover, \eqref{eq:surface-reduction-data} shows
that an estimate
\(\lVert M\rVert\leq A c_1^2(S)\) for surfaces of general type is equivalent
to Gromov's scalar-curvature inequality on this class with coefficient
\(A/(32\pi^2)\).  Proposition~\ref{prop:nef-canonical} gives
\(A=432\pi^2\).

\subsection{Bidisk quotients and the optimal fixed-manifold coefficient}

The uniform estimate above gives the coefficient \(27/2\).  For bidisk
quotients, the fixed-manifold coefficient has a closed formula.  Let
\(\HH^2\) denote the hyperbolic plane of sectional curvature \(-1\).

\begin{proposition}
\label{prop:bidisk}
Let
\[
  X=\Gamma\backslash(\HH^2\times\HH^2)
\]
be a closed quotient with its product complex orientation, where
\(\Gamma<\operatorname{PSL}_2(\mathbb R)\times
\operatorname{PSL}_2(\mathbb R)\) is torsion-free and cocompact.  For
\(r\geq0\), put \(M=X\#r\,\overline{\CP}^{2}\).  If a Riemannian metric
\(g\) on \(M\) satisfies \(\Sc_g\geq-\lambda^2\) for some \(\lambda\geq0\),
then
\begin{equation}
\label{eq:bidisk-bound}
  \lVert M\rVert
  \leq\frac{3}{32\pi^2}\,\lambda^4\vol_g(M).
\end{equation}
The coefficient \(3/(32\pi^2)\) is optimal for the fixed smooth manifold
\(M\).  For \(r=0\), equality is attained by the locally symmetric product
metric.
\end{proposition}

\begin{proof}
Bucher--Karlsson's computation gives
\begin{equation}
\label{eq:bidisk-simplicial}
  \lVert X\rVert=6\chi(X)
\end{equation}
for every closed bidisk quotient, reducible or irreducible
(see \cite{BucherKarlsson_h2xh2_2008}*{Theorem~1}).  Since simplicial
volume is unchanged by blow-up, \(\lVert M\rVert=\lVert X\rVert\).

The product splitting on \(\HH^2\times\HH^2\) is \(\Gamma\)-invariant and
therefore descends to a parallel holomorphic decomposition
\[
  T^{1,0}X=L_1\oplus L_2
\]
into complex line bundles.  The invariant Chern-form representative of
\(c_1(L_i)\) lifts to a form pulled back from the \(i\)-th factor.  Its
square vanishes pointwise, so \(c_1(L_i)^2=0\), including when \(\Gamma\)
is irreducible.  The Whitney product formula gives
\(c_2(T^{1,0}X)=c_1(L_1)c_1(L_2)\).  Hence
\begin{align}
\label{eq:bidisk-chern}
  c_1^2(X)
  &=\int_X\bigl(c_1(L_1)+c_1(L_2)\bigr)^2\notag\\
  &=2\int_Xc_1(L_1)c_1(L_2)
   =2\int_Xc_2(T^{1,0}X)
   =2\chi(X).
\end{align}
Here \(c_1^2(X)=\int_X c_1(K_X)^2\); the square is unchanged when
\(c_1(T^{1,0}X)\) is replaced by its negative \(c_1(K_X)\).

The surface \(X\) is minimal and of general type, hence it is the minimal
model of \(M=X\#r\,\overline{\CP}^{2}\).  LeBrun's formula gives
\begin{equation}
\label{eq:bidisk-yamabe}
  \sigma(M)
  =-4\pi\sqrt{2c_1^2(X)}
  =-8\pi\sqrt{\chi(X)}.
\end{equation}
Combining \eqref{eq:bidisk-simplicial}, \eqref{eq:bidisk-yamabe}, and
Lemma~\ref{lem:yamabe-reformulation}, we obtain
\begin{equation}
\label{eq:bidisk-exact-data}
  \lVert M\rVert=6\chi(X),
  \qquad
  \Acal_4(M)=64\pi^2\chi(X),
  \qquad
  \Rcal_4(M)=\frac{3}{32\pi^2}.
\end{equation}
In particular,
\begin{equation}
\label{eq:bidisk-scalar-cost}
  \Acal_4(M)=\frac{32\pi^2}{3}\lVert M\rVert.
\end{equation}
For every metric \(g\) satisfying \(\Sc_g\geq-\lambda^2\),
\eqref{eq:A4-negative-part} now gives
\begin{align*}
  \lVert M\rVert
  &=\frac{3}{32\pi^2}\Acal_4(M)\\
  &\leq\frac{3}{32\pi^2}\int_M(\Sc_g^-)^2\,d\vol_g\\
  &\leq\frac{3}{32\pi^2}\lambda^4\vol_g(M).
\end{align*}
The definition of \(\Rcal_4(M)\) shows that this coefficient is optimal for
the fixed smooth manifold \(M\).

When \(r=0\), the product metric of sectional curvature \(-1\) on both
factors satisfies
\[
  \Sc=-4,
  \qquad
  \vol(X)=4\pi^2\chi(X),
\]
where the volume identity follows from the Chern--Gauss--Bonnet theorem.
Taking \(\lambda=2\), the right-hand side of
\eqref{eq:bidisk-bound} is \(6\chi(X)=\lVert X\rVert\), so equality
is attained.
\end{proof}

For \(r>0\), optimality refers to the infimum defining \(\Acal_4(M)\) and
does not imply that a minimizing metric exists.  The case \(r=0\) shows
that every coefficient valid uniformly for closed K\"ahler surfaces is at
least \(3/(32\pi^2)\), but it does not determine whether \(27/2\) is sharp.

\begin{remark}
The surface formula is compatible with finite unramified coverings.  Let
\(p:\widetilde S\to S\) be an unramified, equivalently \'{e}tale,
holomorphic covering of degree \(d\) between minimal closed K\"ahler
surfaces of general type.  Since
\(K_{\widetilde S}=p^*K_S\), one has
\(c_1^2(\widetilde S)=dc_1^2(S)\).  Together with multiplicativity of
simplicial volume, \eqref{eq:surface-reduction-data} gives
\begin{equation}
\label{eq:etale-density}
  \Acal_4(\widetilde S)=d\Acal_4(S),
  \qquad
  \lVert\widetilde S\rVert=d\lVert S\rVert,
  \qquad
  \Rcal_4(\widetilde S)=\Rcal_4(S).
\end{equation}
Hence the fixed-manifold coefficient is constant along finite unramified
towers.
\end{remark}
\begin{remark}
The same argument gives a bound for holomorphic surface bundles.  Let
\(\pi:E\to B\) be a holomorphic submersion from a closed K\"ahler surface
to a closed complex curve, and suppose that the base and fiber both have
genus at least two.  Then \(E\) is a minimal surface of general type
(see \cite{Kotschick_regularly_fibered_1999}*{p.~291}).  For every closed
complex surface, the Chern number identity is
\[
  c_1^2(E)=2\chi(E)+3\tau(E),
\]
where \(\tau(E)\) is the signature.  Kotschick's estimate for surface
bundles gives
\[
  3|\tau(E)|<\chi(E)
\]
(see \cite{Kotschick_regularly_fibered_1999}*{Theorem~3}).
Therefore \(c_1^2(E)<3\chi(E)\), and
\begin{equation}
\label{eq:holomorphic-bundle-upper}
  \lVert E\rVert
  \leq432\pi^2c_1^2(E)
  <1296\pi^2\chi(E).
\end{equation}
\end{remark}

\section{A symplectic extension and non-K\"ahler examples}
The proof of Theorem~\ref{thm:main} extends to a class of symplectic
\(4\)-manifolds that do not admit K\"ahler structures.  We first formulate
a general criterion and then construct an infinite family of examples.
% Proof of Theorem \ref{thm:main} can be extended slightly to establish Conjecture \ref{conj:gromov} for many symplectic 4-manifolds which do not admit K\"ahler structures.

As in complex geometry, symplectic \(4\)-manifolds carry a notion of
Kodaira dimension.  Let \((M,\omega)\) be a symplectic \(4\)-manifold.  The
space of \(\omega\)-compatible almost complex structures is nonempty and
contractible, so the class
\[
  K_\omega:=-c_1(TM,J)
\]
is independent of the choice of an \(\omega\)-compatible almost complex
structure \(J\).  It is called the \emph{symplectic canonical class}.  If
\(J\) is integrable, then \(K_\omega\) coincides with $ c_1(K_M)$, the canonical class of teh K\a"hler surface $ (M,\omega,J)$.
\begin{definition}[Symplectic Kodaira dimension]
  For a minimal symplectic 4-manifold $ (M,\omega ) $ with symplectic canonical class $ K_\omega  $, the symplectic Kodaira dimension of $ (M,\omega ) $ is defined as follows:\[
	\kappa^s(M,\omega )=\begin{cases}
	-\infty &\text{if }K_\omega\cdot [\omega]<0 \text{ or}\quad  K_\omega\cdot K_\omega <0\\
	0 &\text{if }  K_\omega\cdot [\omega]=0 \text{ and  } K_\omega\cdot K_\omega=0 \\
	1 &\text{if } K_\omega\cdot [\omega]>0  \text{ and  } K_\omega\cdot K_\omega=0 \\
	2 &\text{if } K_\omega\cdot [\omega]>0 \text{ and  } K_\omega\cdot K_\omega>0
	\end{cases}
	\]
  The Kodaira dimension of a non-minimal symplectic 4-manifold is defined to be that of any of its minimal models.
\end{definition}
Symplectic 4-manifolds with $ \kappa^s=-\infty$ must be rational or ruled (see \cite{Li-Kod}*{Theorem 2.4}). In particular, they must have $ b^+=1$ (see \cite{mcduff-rational}).
If a symplectic 4-manifold has $ \kappa^s\ge 0$, it must have a unique minimal model up to symplectomorphism (see \cite{Li-Kod}*{Proposition 2.1}).
% Symplectic Kodaira dimension actually only depends on the oriented diffeomorphism type of $ M$ (see \cite{Li-Kod}*{Theorem 2.6}).
We say $ (M,\omega)$ is a symplectic 4-manifold of general type if $ \kappa^s(M,\omega)=2$.

For symplectic 4-manifolds of general type, LeBrun's formula for Yamabe invariants is no longer available. Nevertheless, his curvature estimate in terms of Seiberg--Witten monopole classes still provides a lower bound for the scalar cost.

Suppose $ M$ is a smooth compact oriented 4-manifold with $ b^+\ge 2$. Let $ \mathfrak{C}\subset H^2(M;\mathbb{R})$ be the set of \emph{monopole classes}, i.e. the first Chern classes of those spin$^c $ structures on $ M$ for which the Seiberg-Witten equations have solutions
for all metrics. 
There is also a notion of Seiberg-Witten \emph{basic classes}, i.e. the first Chern classes of spin$ ^c$ structures for which the Seiberg-Witten invariant is nonzero. When $ b^+\ge 2$, every basic class is a monopole class.
Consider its convex hull $ \operatorname{Hull}(\mathfrak{C})$.
Define a real-valued invariant of $ M$ by setting \[
\beta^2(M):= \max \{ v^2\, |\, v\in \operatorname{Hull}(\mathfrak{C}) \}\]
when $ \mathfrak{C}$ is nonempty and $ \beta^2(M)=0$ if $ \mathfrak{C}$ is empty.
LeBrun proved in \cite{LeBrun_convex}*{Theorem A} that any metric $ g$ satisfies the curvature estimate \[\int_M \Sc_g^2 \, d\vol_g \ge 32\pi^2\beta^2(M).
\]
Although this estimate doesn't directly imply the same estimate for $ \Sc_g^-$, the same proof shows \begin{equation}\label{eq:LeBrun-Sc-ineq}
  \mathcal{A}_4(M)=\inf_g \int_M (\Sc^-_g)^2 \, d\vol_g \ge 32\pi^2\beta^2(M).
\end{equation}

The next proposition gives a criterion under which the conclusion of
Theorem~\ref{thm:main} extends to a symplectic \(4\)-manifold of general
type.
\begin{proposition}\label{prop:symplectic}
  Let $ (M,\omega)$ be a symplectic 4-manifold of general type with $ b^+(M)\ge 2$. Suppose there is a minimal K\"ahler surface $ S$ so that \begin{itemize}
    \item $ S$ has the same $ 2\chi +3\tau$ as the minimal model $ M_{min}$ of $ M$;
    \item $ \lVert S\rVert \ge \lVert M_{min} \rVert $.
  \end{itemize}
  Then every Riemannian metric \(g\) on \(M\) satisfies
\begin{equation}
  \lVert M\rVert
  \leq \frac{27}{2}
  \int_M (\Sc_g^-)^2\,d\vol_g.
\end{equation}
In particular, if \(\Sc_g\geq-\lambda^2\) for some \(\lambda\geq0\), then
\begin{equation}
  \lVert M\rVert
  \leq \frac{27}{2}\,\lambda^4\vol_g(M).
\end{equation}
  
  % Theorem \ref{thm:main} holds if its minimal model has the same $ 2\chi +3\tau$ as a minimal K\"ahler surface.
  % Let $ (M,\omega)$ be a symplectic 4-manifold with $ b^+\ge 2$ and symplectic Kodaira dimension $ \kappa^s(X)=2$. Suppose the minimal model $ M_{min}$ is oriented homotopy equivalent to a minimal K\"ahler surface of general type. Then every Riemannian metric $ g$ on $ M$ satisfies
\end{proposition}
\begin{proof}
  The case where $ \lVert M \rVert=0 $ is trivially true. So in the following we may assume $ \lVert M \rVert>0$.
  Let $ (M_{min},\omega_{min})$ be a minimal model and $ \pi:M\to M_{min}$ be the blow-up map. Then $ M\cong M_{min}\# l\overline{\CP^2}$ for some $ l\ge 0$.  The symplectic canonical class is
  \[K_\omega = \pi^* K_{min} + \sum_{i=1}^l E_i,\]
  where $ K_{min}:=K_{\omega_{min}}$ and $ \{E_i\}$ are exceptional classes.
  Since $ b^+(M_{min})=b^+(M)\ge 2$, $ K_{min}$ is a basic class and therefore a monopole class of $ M_{min}$ (see \cite{Taubes-SW-symplectic}).
  By Seiberg-Witten blow-up formula (see \cite{fintushel-stern-immersed}*{Theorem 1.4}), we know that $ \pi^*K_{min}+\sum_{i=1}^l E_i $ and $ \pi^*K_{min}-\sum_{i=1}^l E_i$ are both monopole classes. Hence $ \pi^*K_{min}\in \operatorname{Hull}(\mathfrak{C})$. 
  
  By our hypothesis, there is a minimal K\"ahler surface $ S$ such that  \[K_{min}^2=2\chi(M_{min})+3\tau(M_{min})=2\chi(S)+3\tau(S)=K_S^2. \]
  Since $ \lVert S \rVert \ge \lVert M\rVert >0 $, we know $ S$ is of general type by Lemma \ref{lem:surface-reduction} and then $ K_S$ is nef and big.
  By \eqref{eq:LeBrun-Sc-ineq} and \eqref{eq:SV_boundby_K}, we have 
  \begin{align*}
  \dfrac{27}{2}\int_M (\Sc^-_g)^2\, d\vol_g &\ge 432\pi^2 (\pi^*K_{min})^2=432\pi^2 K_S^2\\
  &\ge  \lVert S\rVert \ge \lVert M_{min}\rVert =\lVert M \rVert.
  \end{align*}
  In particular, if $ \Sc_g\ge -\lambda^2$ for some $ \lambda\ge 0$, then \[
  \lVert M\rVert \le \dfrac{27}{2}\lambda^4 \vol_g(M).\]
\end{proof}

We now construct an infinite family of non-K\"ahler symplectic
\(4\)-manifolds to which Proposition~\ref{prop:symplectic} applies.
\begin{proposition}\label{prop:symplectic-exmaple}
  For each pair of integers $ g,h\ge 2$, there is an infinite family of minimal symplectic 4-manifolds $ \{X_i\}$ of general type with $ b^+\ge 2$ and a family of minimal K\"ahler surfaces $ \{S_i \}$, such that for all sufficiently large $ i$,
  \begin{itemize}
    \item $ X_i$ admits no K\"ahler structure,
    \item $ \pi_1(X_i)$ is nonamenable,
    \item $ 2\chi(X_i)+3\tau(X_i)=2\chi(S_i)+3\tau(S_i)$,
    \item $ 0<\lVert X_i\rVert \le \lVert S_i \rVert$.
  \end{itemize}
\end{proposition}
\begin{proof}
  \smallskip
  \noindent\emph{Step 1: The first building block.}
  % \textbf{(1) First building block} $ E$:\\ 
  Fix integers $ g,h\ge 2$ and consider the surface bundle 
  \[ F\to E\to B, \]
  where $ F=\Sigma_g$, $ B=\Sigma_h$. Choose standard generators $ \alpha_1,\beta_1,\dots,\alpha_h,\beta_h$ for $ \pi_1(B)$. Choose a nonseparating curve $ \delta\subset F$ and define the monodromy $ \rho:\pi_1(B)\to \pi_0\operatorname{Diff}(F)$ by \[
  \rho(\alpha_1)=D_\delta, \quad \rho(\beta_1)=\dots =\rho(\alpha_h)=\rho(\beta_h)=1,\]
  where $ D_\delta$ denotes the positive Dehn twist along $ \delta$.
  The bundle admits a section of self-intersection $ 0$ as the monodromy fixes a disk away from $ \delta$. So the fiber class is nonzero and there is a symplectic structure $ \omega_E$ on $ E$ by Thurston's construction (see \cite{Thurston-symplectic}). The total space $ E$ is aspherical since both $ \Sigma_g$ and $ \Sigma_h$ are aspherical. Hence $ E$ is minimal.
  Its relevant topological invariants can be computed as follows \begin{itemize}
    \item $ \chi(E)=\chi(F)\chi(B)=4(g-1)(h-1)$.
    \item $ \tau(E)=0$ since the monodromy group is abelian and thus amenable (see \cite{Morita-vanishing}).
    \item Its first homology is \[H_1(E;\mathbb{Q})=H_1(\Sigma_g;\mathbb{Q})/\langle \beta -(D_\delta)_*\beta \,|\, \beta\in H_1(\Sigma_g;\mathbb{Q})\rangle \oplus \bigoplus_{i=1}^{2h} H_1(S^1;\mathbb{Q}),\]i.e. $ b_1(E)=2g+2h-1$.
    \item Its simplicial volume satisfies \[ \lVert E\rVert \ge \lVert F\rVert\cdot \lVert B\rVert =16(g-1)(h-1).\]
  \end{itemize}
  There are product Lagrangian tori \[
  T=x\times \alpha_2,\quad T^*=y\times \beta_2,\]where $ x,y\subset F$ are simple closed curves disjoint from $ \delta$ and intersecting each other transversely at exactly one point.
  $ T\cdot T^*=1 $ and thus $ [T]\neq 0$. So we can perturb $ T$ into a symplectic torus, which we still denote by $ T$.
  The meridian of $ T$ bounds $ T^*\setminus \nu T$ in $ A:=E\setminus \operatorname{int}\nu T$, so the meridian is null-homologous in $ A$ and $ b_1(A)=b_1(E)$.

  \smallskip
  \noindent\emph{Step 2: The second building block.}
  A triple $ (C,T_1,T_2)$ is called a telescoping triple if \begin{itemize}
    \item $ C$ is a symplectic 4-manifold, and $ T_1, T_2$ are disjoint embedded Lagrangian tori in $ C$ such that $ [T_1],[T_2]$ span a two dimensional subspace of $ H_2(C;\mathbb{R})$.
    \item $ \pi_1(C)\cong \mathbb{Z}^2$ and $ \pi_1(C\setminus(T_1\cup T_2))\to \pi_1(C)$ induced by inclusion is an isomorphism.
    \item The image of $ \pi_1(T_1)\to \pi_1(C)$ is a summand $ \mathbb{Z}\subset \pi_1(C)$.
    \item $ \pi_1(T_2)\to \pi_1(C)$ is an isomorphism.
  \end{itemize}
  A telescoping triple was constructed in \cite{ABBKP}*{Theorem 11} such that $ C$ is minimal, $ \chi(C)=8$ and $ \tau(C)=-4$.
  Perturb both tori to become symplectic and of the same area. Consider \[Q_N:=\underbrace{C\# \dots \# C}_{2N\text{ copies}},\]
  where each $ \#$ denotes a symplectic sum along $ T_2$ of one copy and $ T_1$ of the next copy. Since $ b^+(C)=3$, it is not rational or ruled and therefore $ Q_N$ is minimal (see \cite{Usher-minimality}*{Theorem 1.1}) and has $ \chi(Q_N)=16N$ and $ \tau(Q_N)=-8N$. Let $ T_1^1$ denote the $ T_1$ in the first copy of $ C$, $ U_N$ denote the $ T_2$ in the last copy, and $ B_N:=Q_N\setminus \operatorname{int} \nu U_N$. Then $ (C,T_1^1, U_N)$ is again a telescoping triple by \cite{ABBKP}*{Proposition 3}.
  We then have $ \pi_1(B_N)\cong \mathbb{Z}^2$ and $ \pi_1(\partial B_N)\cong \mathbb{Z}^3$. Since both fundamental groups are amenable, we have the relative bounded cohomology $ H_b^4(B_N,\partial B_N)=0$ and therefore $ \lVert B_N,\partial B_N\rVert=0$ (see \cite{relative-sv}).

  \smallskip
  \noindent\emph{Step 3: The symplectic sum construction.}
  Perturb the symplectic structure locally near $ U_N$ or $ T$ to make $ U_N\subset Q_N$ and $ T\subset E$ to have the same symplectic area. 
  Combining the two pieces $ E$ and $ Q_N$ via symplectic sum (see \cite{Gompf-sum}), we get the desired symplectic manifold \[
  (X_N,\omega_N)=E\#_{T=U_N} Q_N=A\cup_\psi B_N,\]
  where $ \psi:\partial B_N\to -\partial A$ is the orientation preserving and meridian preserving gluing map. Its relevant topological invariants are as follows. 
  \begin{itemize}
    \item $ \chi(X_N)=\chi(E)+\chi(Q_N)=4(g-1)(h-1)+16N$.
    \item $ \tau(X_N)=\tau(E)+\tau(Q_N)=-8N$.
    \item $ b_1(X_N)=b_1(A)+b_1(B_N)-2=2g+2h-1$ by Mayer-Vietoris. 
    \item $ b_2(X_N)=4gh+16N$ and $ b^+(X_N)=2gh+4N$.
  \end{itemize}
  In particular, \(X_N\) admits no K\"ahler structure because
\(b_1(X_N)\) is odd.
  $ X_N$ is minimal since both $ E$ and $ Q_N$ are minimal and both have $ b^+>1$.
  $ X_N$ cannot have $ \kappa^s(X_N)=-\infty$ since $ b^+(X_N)> 1$.
  \[
  K_{\omega_N}^2=2\chi(X_N)+3\tau(X_N)=8(g-1)(h-1)+8N>0\]
  implies that $ X_N$ is of general type.
  Since $ \pi_1(\partial A)$ is amenable, by Gromov's additivity theorem (see \cite{gromov_bounded_cohomology_1982}) we get \[ \lVert X_N \rVert \le \lVert A,\partial A \rVert + \lVert B_N,\partial B_N \rVert\le \lVert A,\partial A \rVert, \]
  where this upper bound is independent of $ N$.

  \smallskip
\noindent\emph{Step 4: Positivity of simplicial volume and comparison with
K\"ahler surfaces.}
  \begin{lemma}
    There is a degree $ 1$ map \[f_N:X_N\to E.\]
  \end{lemma}
  \begin{proof}
    Let $ V:=\nu T$ be the neighborhood of $ T$ in $ E$ such that $ E=A\cup V$, i.e. $\partial V=-\partial A$. Since $ X_N=A\cup_\psi B_N$, we may view the gluing map as $ \psi:\partial B_N\to \partial V$. Denote the inclusion maps $ i:\partial B_N\to B_N$ and $ j:\partial V\to V$. 
    Since $ (Q_N,T_1^1,U_N)$ is a telescoping triple, we must have $ \ker i_*$ is generated by the meridian.
    Note $ \ker (j_*\circ \psi_*)$ is also generated by the meridian as $ \psi_*$ preserves meridian. So $ j_*\circ \psi_*$ factors uniquely through $ i_*$ and there is a homomorphism $ \Phi:\pi_1(B_N)\to \pi_1(V)$ so that the following diagram commutes.
    % https://tikzcd.yichuanshen.de/#N4Igdg9gJgpgziAXAbVABwnAlgFyxMJZABgBpiBdUkANwEMAbAVxiRAB120sB9ARgAUnNHQBOeRgAIAQjwByAShABfUuky58hFH3JVajFm2G9BsxSrUgM2PASJk+++s1aIOXU0K5iJDSQBqSqrqtlpEuk7ULkbuJvwCQSr6MFAA5vBEoABmohAAtkhkIDgQSABM1Ax0AEYwDAAKGnbaIKJYaQAWOCDRhm4eaNg8AFSWOXmFiMWlSLoGrmy8YyEguQVz1LOIAMx9i3HsDZ1Y42uTFVtlu1W19U1h9u7tXT37sSAAVqPJykA
  \[\begin{tikzcd}
  \pi_1(\partial B_N) \arrow[d, "\psi_*"'] \arrow[r, "i_*"] & \pi_1(B_N) \arrow[d, "\Phi"] \\
  \pi_1(\partial V) \arrow[r, "j_*"']                       & \pi_1(V)                    
  \end{tikzcd}\]
  Since $ V\simeq T^2$ is a $ K(\mathbb{Z}^2,1)$-space, the homotopy classes of maps from any $ M$ to $ V$ are given by \[
  [M,V]=\operatorname{Hom}(\pi_1(M),\mathbb{Z}^2).\]
  In particular, there is a map $ \phi:B_N\to V$ such that $ \phi_*=\Phi$.
  We must also have $ \phi|_{\partial B_N}\simeq j\circ \psi$ since $ \Phi\circ i_*=j_*\circ \psi_*$.
  % \in \operatorname{Hom}(\pi_1(\partial B_N),\mathbb{Z}^2)$.
  Let $ H:\partial B_N\times [0,1]\to V$ be a homotopy such that $ H_0=\phi|_{\partial B_N}$ and $ H_1=j\circ \psi$. By the homotopy extension property of the pair $ (B_N,\partial B_N)$, we get a homotopy $ \tilde{H}:B_N\times [0,1]\to V$ such that $ \tilde{H}_0=\phi$ and $ \tilde{H}(x,t)=H(x,t)$ for $ x\in \partial B_N$. Set $ g:=\tilde{H}_1:B_N\to V$ we have $ g|_{\partial B_N}=j\circ \psi$. Therefore $ g$ glues with $ \operatorname{id}:A\to A$ to give a map $ f_N:X_N\to E$.

  Since $ f_N$ is identity on $ A$, its degree must be $ 1$.
  \end{proof}
  By monotonicity of simplicial volume, \cite{Bucher-fiberbundle}*{Corollary 1.3} and \cite{BucherKarlsson_h2xh2_2008}*{Corollary 3}, we get \[\lVert X_N\rVert\ge \lVert E\rVert \ge \lVert \Sigma_g \times \Sigma_h \rVert =24(g-1)(h-1)>0. \]
  In particular, $ \pi_1(X_N)$ is nonamenable.
  Let $ N_i=i(g-1)$ and $ X_i:=X_{N_i}$. Consider the K\"ahler surface \[S_i= \Sigma_g\times \Sigma_{h+i}, \]
  which has \begin{equation}
    2\chi(S_i)+3\tau(S_i)=8(g-1)(h+i-1)=2\chi(X_i)+3\tau(X_i).
  \end{equation} 
  For all $ i$ sufficiently large, we get \[
  \lVert S_i \rVert =24(g-1)(h+i-1)>\lVert A,\partial A \rVert \ge \lVert X_i \rVert. \]
\end{proof}

\begin{corollary}
\label{cor:nonkahler-gromov}
For each pair of integers \(g,h\geq2\), the manifolds \(X_i\) constructed
in Proposition~\ref{prop:symplectic-exmaple} satisfy, for all sufficiently
large \(i\),
\[
  \lVert X_i\rVert
  \leq\frac{27}{2}
  \int_{X_i}(\Sc_g^-)^2\,d\vol_g
\]
for every Riemannian metric \(g\) on \(X_i\).  In particular, they satisfy
Conjecture~\ref{conj:gromov} with the same constant as in
Theorem~\ref{thm:main}.
\end{corollary}

\section{Canonical volume and scalar cost in higher dimensions}
\label{sec:higher-limit}

Proposition~\ref{prop:nef-canonical} holds in every complex dimension.  A
scalar-curvature consequence would also require canonical volume to be
controlled by scalar cost.  For surfaces of general type, LeBrun's formula
gives the exact identity
\[
  \Acal_4(M)=32\pi^2c_1^2(S),
\]
where \(S\) is the minimal model.  In complex dimension at least three, no
such positive lower bound holds, even for ample canonical bundles.

\begin{proposition}
\label{prop:higher-yamabe-failure}
For every \(m\geq3\), there exists a smooth projective \(m\)-fold \(X\)
with ample canonical bundle such that
\begin{equation}
\label{eq:higher-counterexample-data}
  \int_X c_1(K_X)^m=m+3>0,
  \qquad
  \sigma(X)>0,
  \qquad
  \Acal_{2m}(X)=\lVert X\rVert=0.
\end{equation}
In particular, there is no constant \(b_m>0\) such that
\begin{equation}
\label{eq:false-higher-bridge}
  \Acal_{2m}(X)
  \geq b_m\int_X c_1(K_X)^m
\end{equation}
for every closed K\"ahler \(m\)-fold with nef and big canonical bundle.
\end{proposition}

\begin{proof}
Following LeBrun, let
\[
  X=X_{m+3}\subset\CP^{m+1}
\]
be a smooth hypersurface of degree \(d=m+3\)
(see \cite{LeBrun_kodaira_yamabe_1999}*{Introduction, p.~134}).
Let \(H=c_1(\mathcal{O}_{\CP^{m+1}}(1)|_X)\) be the restricted hyperplane
class.  The adjunction formula gives
\[
  K_X\cong\mathcal{O}_X(d-m-2)=\mathcal{O}_X(1).
\]
Hence \(K_X\) is ample, \(c_1(K_X)=H\), and
\[
  \int_X c_1(K_X)^m=\int_X H^m=d=m+3.
\]

The Lefschetz hyperplane theorem gives
\(\pi_1(X)=0\) and
\(H^2(X;\mathbb Z)=\mathbb Z\langle H\rangle\).  Moreover
\[
  w_2(X)\equiv c_1(T^{1,0}X)
  \equiv(m+2-d)H
  \equiv H\pmod2.
\]
Since \(H\) is a generator, its reduction modulo \(2\) is nonzero, and
therefore \(X\) is not spin.  Its real dimension is \(2m\geq6\), so the
Gromov--Lawson classification of simply connected non-spin manifolds
provides a metric of positive scalar curvature
(see \cite{GL_simply_psc}*{Corollary~C, p.~424}).  Hence
\(\sigma(X)>0\), and Lemma~\ref{lem:yamabe-reformulation} gives
\(\Acal_{2m}(X)=0\).

Finally, \(\pi_1(X)\) is trivial and therefore amenable.  Gromov's mapping
theorem gives \(\lVert X\rVert=0\)
(see \cite{gromov_bounded_cohomology_1982}*{Section~3.1,
Corollary~\textup{(C)}}).  This proves
\eqref{eq:higher-counterexample-data}.
\end{proof}

Here both scalar cost and simplicial volume vanish, so Gromov's conjecture
is unaffected.  Positive canonical volume alone cannot give a positive
lower bound for scalar cost.

Suppose that \(K_X\) is ample.  Apply
Lemma~\ref{lem:twisted-negative} with \(\theta=0\) and
\[
  [\omega_h]=\frac{2\pi}{2m-1}c_1(K_X)
\]
to obtain a K\"ahler--Einstein metric \(h\) satisfying
\(\Ric_h=-(2m-1)h\).  Its scalar curvature and volume are
\[
  \Sc_h=-2m(2m-1),
  \qquad
  \vol_h(X)
  =\frac1{m!}
   \left(\frac{2\pi}{2m-1}\right)^m
   \int_X c_1(K_X)^m.
\]
Since \(h\) has negative constant scalar curvature, it minimizes the Yamabe
functional in its conformal class.  Hence
\begin{equation}
\label{eq:ke-yamabe}
  Y(X,[h])
  =-\frac{4\pi m}{(m!)^{1/m}}
  \left(\int_X c_1(K_X)^m\right)^{1/m}.
\end{equation}
The smooth Yamabe invariant is the supremum over all conformal classes, so
\eqref{eq:ke-yamabe} gives only the lower bound
\[
  \sigma(X)\geq Y(X,[h]).
\]
On the other hand, Lemma~\ref{lem:yamabe-reformulation} shows that an
estimate
\[
  \Acal_{2m}(X)\geq b_m\int_X c_1(K_X)^m
\]
would require the upper bound
\[
  \sigma(X)
  \leq-b_m^{1/m}
  \left(\int_X c_1(K_X)^m\right)^{1/m}<0.
\]
LeBrun's Seiberg--Witten argument gives this reverse inequality in complex
dimension two, whereas Proposition~\ref{prop:higher-yamabe-failure} rules
it out for \(m\geq3\).  Thus only in complex dimension two does canonical
volume give the scalar-cost estimate used above.

\section{Questions on sharp constants}
\label{sec:directions}

Define the optimal uniform coefficient for closed K\"ahler surfaces by
\begin{equation}
\label{eq:optimal-kahler-constant}
  c_4^{\mathrm{Kah}}
  :=\sup\left\{
  \frac{\lVert M\rVert}{\Acal_4(M)}:
  M\text{ is a closed K\"ahler surface},\
  \Acal_4(M)>0
  \right\}.
\end{equation}
The results above give
\begin{equation}
\label{eq:kahler-constant-interval}
  \frac{3}{32\pi^2}
  \leq c_4^{\mathrm{Kah}}
  \leq\frac{27}{2}.
\end{equation}
The upper bound comes from the general Ricci-curvature estimate in
Proposition~\ref{prop:nef-canonical}; the bidisk family attains the lower bound.

\begin{problem}
Determine \(c_4^{\mathrm{Kah}}\).  In particular, decide
whether the bidisk coefficient \(3/(32\pi^2)\) remains valid for every
closed K\"ahler surface of general type.
\end{problem}

A second question concerns the sharpness of the nef-canonical estimate.
Its constant comes from Gromov's general Ricci-curvature inequality, and the
K\"ahler condition may permit a smaller universal constant.  In complex
dimension one the optimal value is
\(C_1^{\mathrm{can}}=2\), because a curve of genus \(g\geq2\) satisfies
\(\lVert X\rVert=4g-4=2\deg K_X\).

\begin{problem}
For each \(m\geq2\), find the optimal constant \(C_m^{\mathrm{can}}\) such
that
\[
  \lVert X\rVert
  \leq C_m^{\mathrm{can}}
  \int_X c_1(K_X)^m
\]
for every closed K\"ahler \(m\)-fold with nef canonical bundle.
\end{problem}

Proposition~\ref{prop:higher-yamabe-failure} rules out only estimates that
factor through canonical volume; higher-dimensional bounds involving other
smooth or index-theoretic invariants remain possible.

Finally, consider the simplicial volumes realized in the K\"ahler category.
Set
\[
  \operatorname{SV}_{\mathrm{Kah}}(4)
  :=\bigl\{\lVert X\rVert:
  X\text{ is a closed connected K\"ahler surface}\bigr\}.
\]
Every nonnegative rational number is the simplicial volume of a closed
oriented connected four-manifold
(see \cite{HeuerLoeh_spectrum_2021}*{Theorem~B}), but the construction of
Heuer and L\"oh does not yield K\"ahler surfaces.

\begin{problem}
Determine \(\operatorname{SV}_{\mathrm{Kah}}(4)\).  In particular, does it
contain every nonnegative rational number?  More generally, given a closed
oriented connected manifold \(M\), does there exist a closed connected
K\"ahler surface \(X\) such that
\[
  \lVert X\rVert=\lVert M\rVert?
\]
\end{problem}

\renewcommand{\biblistfont}{\normalfont\fontsize{10}{11}\selectfont}
\bibliographystyle{amsalpha}
\bibliography{references}

\end{document}